\documentclass[a4paper,12pt]{article}

\usepackage{amsfonts,amssymb,amsmath,graphicx}
\usepackage{theorem}
 {\theorembodyfont{\rm}
  \newtheorem{defi}{Definition}[section]
  \newtheorem{rem}[defi]{Remark}
  \newtheorem{exa}[defi]{Example}
  
 }

  \newtheorem{thm}[defi]{Theorem}
  
\newcommand{\Aa}{{\mathbb A}}

\newcommand{\CC}{{\mathbb C}}

\newcommand{\HH}{{\mathbb H}}

\newcommand{\PP}{{\mathbb P}}

\newcommand{\RR}{{\mathbb R}}

\newcommand{\cB}{{\mathcal B}}
\newcommand{\cC}{{\mathcal C}}
\newcommand{\cD}{{\mathcal D}}

\newcommand{\cR}{{\mathcal R}}

\newcommand{\fC}{{\mathfrak C}}

\newcommand{\fL}{{\mathfrak L}}

\newcommand{\End}{{\mathrm{End}}}

\newcommand{\T}{{\mathrm{T}}}

\newcommand{\id}{{\mathrm{id}}}
\newcommand{\dis}
                 {{\mathrel{\scriptstyle{\triangle}}}}

\newcommand{\GL}{{\mathrm{GL}}}

\newcommand{\E}{{\mathrm{E}}}

\newcommand{\dist}{{\mathrm{dist}}}
\newcommand{\diag}{{\mathrm{diag}}}
\newcommand{\Rd}{{\widehat{R}^2}}
\newcommand{\Pd}{{\widehat\PP}}
\let\phi=\varphi

\let\theta=\vartheta

\newcommand{\SDelimArray}[4]{\hbox{\scriptsize\arraycolsep=.5\arraycolsep
  $\left#1\!\!\begin{array}{*{#3}{c}}#4\end{array}\!\!\right#2$}}

\newcommand{\SMat}{\SDelimArray()}

\newenvironment{proof}
    {\begin{trivlist} \item {\sl Proof:}} 
    {\/ $\square$ \end{trivlist}}

\newcounter{abbildung} 
\begin{document}
\title{The Dual of a Chain Geometry}
\author{Andrea Blunck\thanks{Supported by a Lise Meitner
Research Fellowship of the Austrian Science Fund (FWF), project M574-MAT.} \and
Hans Havlicek}


\maketitle

\begin{abstract}
We introduce and discuss the dual of a chain geometry. Each chain geometry is
canonically isomorphic to its dual.  This allows us to show that there are
isomorphisms of chain geometries that arise from antiisomorphisms of the
underlying rings.


\end{abstract}

\section{Introduction}

For each left module over a ring $R$ there is the dual module. It may
be considered as a right $R$-module or as a left module over the
opposite ring $R^\circ$. A chain geometry $\Sigma(K,R)$ is based upon
a proper subfield $K$ of a ring $R$ and the left $R$-module $R^2$.
Observe that we do not assume that $K$ is in the centre of $R$. The
\emph{dual chain geometry} $\widehat\Sigma(K,R)$ of $\Sigma(K,R)$ is
defined via the dual module of $R^2$. Up to notation
$\widehat\Sigma(K,R)$ is the same as the chain geometry
$\Sigma(K^\circ,R^\circ)$. There is a ``canonical isomorphism'' from
each chain geometry onto its dual. However, in general it seems
difficult to describe it explicitly for all points in terms of
coordinates unless the underlying ring $R$ has some additional
properties.

We establish that each residue of a chain geometry can be identified
with a residue of its dual in a natural way. However, the
(algebraically defined) relation of compatibility given on the set of
blocks of each residue is not always preserved under the canonical
isomorphism, whence one obtains also a notion of \emph{dual
compatibility}.

From \cite[Theorem 4.1]{blu+h-99a}, the point set of each residue
together with one compatibility class of blocks forms a partial
affine space which is embeddable in the affine space on the left
vector space $R$ over $K$. This result remains true if
``compatibility'' and ``left vector space'' are replaced with ``dual
compatibility'' and ``right vector space'', respectively. In
addition, we give an example of a chain geometry with the following
property: The point set of a residue together with certain blocks of
different compatibility classes forms not only a partial affine space
but a non-desarguesian affine plane.

Finally, we show that two chain geometries $\Sigma(K,R)$ and
$\Sigma(K',R')$ are isomorphic if there is an antiisomorphism $R\to
R'$ that takes $K$ onto a subfield of $R'$ which is conjugate to
$K'$. Again, an explicit description in terms of coordinates of such
an isomorphism of chain geometries does not seem at hand for
arbitrary rings, but we are able to give a formula which allows to
calculate the images of all points in the connected component of the
point $R(1,0)$. This generalizes, in part, a result on isomorphisms
of chain geometries in \cite{bart-89}.

\section{Preliminaries}\label{se:prelim}

Throughout this paper we shall only consider associative rings with a
unit element $1$, which is preserved by homomorphisms, inherited by
subrings, and acts unitally on modules.  The group of invertible
elements of a ring $R$ will be denoted by $R^*$. We refer to
\cite[Chapter II]{bour-89} for the basic properties of free modules.

Consider the free left $R$-module $R^2$ and the group $\GL_2(R)$ of
invertible $2\times 2$-matrices with entries in $R$. A pair $(a,b)\in
R^2$ is called \emph{admissible}, if there exists a matrix in
$\GL_2(R)$ with $(a,b)$ being its first row. The \emph{projective
line over\/} $R$ is the orbit of the free cyclic submodule $R(1,0)$
under the action of $\GL_2(R)$. In other words, $\PP(R)$ is the set
of all $p\leq R^2$ such that $p=R(a,b)$ for an admissible pair
$(a,b)\in R^2$; compare \cite[p.~785]{herz-95}. From
\cite[Proposition~2.1]{blu+h-00a}, in certain cases $R(x,y)\in\PP(R)$
does not imply the admissibility of $(x,y)\in R^2$. However, we adopt
the convention that points are represented by admissible pairs only.
Two admissible pairs represent the same point exactly if they are
left-proportional by a unit in $R$.

Points $p=R(a,b)$ and $q= R(c,d)$ are called \emph{distant} if
$\SMat2{a&b\\c&d}\in \GL_2(R)$. The vertices of the \emph{distant
graph} on $\PP(R)$ are the points of $\PP(R)$, the edges of this
graph are the unordered pairs of distant points. The set $\PP(R)$ can
be decomposed into \emph{connected components} (maximal connected
subsets of the distant graph), for each connected component there is
a \emph{distance function} ($\dist(p,q)$ is the minimal number of
edges needed to go from vertex $p$ to vertex $q$). All connected
components share a common \emph{diameter} (the supremum of all
distances between its points). See \cite[Theorem~3.2]{blu+h-00c}.

Let $K\subset R$ be a (not necessarily commutative) proper subfield.
The projective line over $K$ can be embedded in $\PP(R)$ via
$K(k,l)\mapsto R(k,l)$. The image of $\PP(K)$ under this embedding is
a subset $\cC\subset\PP(R)$ called the \emph{standard chain}. The
orbit of $\cC$ under the action of $\GL_2(R)$ is denoted by
$\fC(K,R)$ and each of its elements is called a \emph{chain}.
Altogether the \emph{chain geometry} $\Sigma(K,R)$ is the incidence
structure with point set $\PP(R)$ and chain set $\fC(K,R)$
\cite{blu+h-99a}. Observe that a chain geometry according to this
definition has been called a \emph{generalized chain geometry} in
\cite{blu+h-00a} and \cite{blu+h-00c} in order to distinguish from an
``ordinary'' chain geometry where $K$ is in the centre of $R$.
However, in the present paper such a distinction will not be
essential.

\section{The Dual of a Chain Geometry}

Reversing the multiplication in the ring $R$ yields the
\emph{opposite ring} $R^\circ$ and the projective line
$\PP(R^\circ)$. Further, if $K$ is a proper subfield of $R$, then the
opposite field $K^\circ$ appears as a proper subfield of $R^\circ$
and we obtain the chain geometry $\Sigma(K^\circ,R^\circ)$. The left
$R^\circ$-module $(R^\circ)^2$ can be considered as a \emph{right}
$R$-module in a natural way. It will then be denoted by $\Rd$ and its
elements will be written as \emph{columns} rather than rows. The
right $R$-module $\Rd$ will be identified with the \emph{dual module}
of $R^2$ as usual, i.e., the image of $(a,b)\in R^2$ under
$(v,w)^\T\in\Rd$ is given by their matrix product. For a subset
$U\subset R^2$ we write
\begin{equation}\label{eq:orthogonal}
  U^\bot:=\left\{\SMat2{x\\y}\in\Rd\mid \forall(a,b)\in U:
  (a,b)\cdot\SMat2{x\\y}=0\right\}.
\end{equation}
Furthermore, we have
  \begin{equation}\label{eq:U-perp}
 (U\cdot M)^\perp=M^{-1}\cdot U^\perp
  \end{equation}
for all $M\in\GL_2(R)$ and all $U\subset R^2$.

By changing from $(R^\circ)^2$ to $\Rd$ we obtain the \emph{dual
projective line} $\Pd(R)$ of $\PP(R)$ as alternative algebraic
description of the projective line $\PP(R^\circ)$. So an element of
$\Pd(R)$ has the form $M\cdot (1,0)^\T R$, with $M\in\GL_2(R)$.
Similarly, one obtains $\widehat\Sigma(K,R)$, the \emph{dual chain
geometry} of $\Sigma(K,R)$. Its set of chains is written as
$\widehat\fC(K,R)$. Since the module $R^2$ is free, it can be
identified with its bidual module. Up to this identification, the
dual of $\widehat\Sigma(K,R)$ is again $\Sigma(K,R)$.

\begin{thm}\label{thm:iota}
Let\/ $\Sigma(K,R)$ be a chain geometry. Then the  mapping
  \begin{equation}\label{eq:iota}
 \iota : \PP(R)\to\Pd(R) : p\mapsto p^\bot
\end{equation}
is an isomorphism of\/ $\Sigma(K,R)$ onto its dual.
\end{thm}

\begin{proof}
Obviously,
\begin{equation}\label{eq:iota-0}
  (R(1,0))^\iota=\SMat2{0 \\ 1} R,
\end{equation}
and $R(1,0)$ is the only $\iota$-preimage of $(0,1)^\T R$. Each point
$p\in\PP(R)$ can be written in the form $p=R(1,0)\cdot M$ with
$M\in\GL_2(R)$. So (\ref{eq:U-perp}) implies that
$p^\iota=M^{-1}\cdot(0,1)^\T R\in\Pd(R)$ and that $\iota$ is
bijective. Let
\begin{equation}\label{eq:E}
  E(t):=\SMat2{t&1\\-1&0} \mbox{ with } t\in R.
\end{equation}
Then $E(t)\in\GL_2(R)$ with
\begin{equation}\label{eq:E-inv}
  E(t)^{-1}=\SMat2{0&-1\\1&t}=E(0)\cdot E(-t)\cdot E(0).
\end{equation}
Hence (\ref{eq:U-perp}) and (\ref{eq:iota-0}) imply
\begin{equation}\label{eq:iota-1}
   (R(t_1,1))^\iota=(-1,t_1)^\T R
\end{equation}
for all $t_1\in R$. So $\iota$ maps the standard chain
$\cC=\{R(k,1)\mid k\in K \}\cup\{R(1,0)\}\in\fC(K,R)$  onto
$\{(-1,k)^\T R\mid k\in K \} \cup \{(0,1)^\T R\}$, which is the
standard chain in $\widehat\Sigma(K,R)$. From the definition of
chains and (\ref{eq:U-perp}), the mapping $\iota$ yields a bijection
of $\fC(K,R)$ onto $\widehat\fC(K,R)$.
\end{proof}

We refer to  $\iota$ as the \emph{canonical isomorphism}
$\Sigma(K,R)\to\widehat\Sigma(K,R)$.

\begin{rem}
Each point $p\in\PP(R)$ is spanned by the first row of a matrix
$M\in\GL_2(R)$. From (\ref{eq:U-perp}) and (\ref{eq:iota-0}) it
follows that  $p^\iota$ is spanned by the second column of $M^{-1}$.
Thus, whenever one has an algorithm to invert matrices of $\GL_2(R)$
then it is also possible to calculate explicitly the $\iota$-image of
a point given in that form. For example, when $R$ is commutative then
$R(a,b)^\iota=(-b,a)^\T R$ for all admissible pairs $(a,b)\in R^2$.
\end{rem}

\begin{rem}\label{rem-iota}
We recall that the \emph{elementary subgroup\/} $\E_2(R)$ of
$\GL_2(R)$ is generated by the set of all matrices (\ref{eq:E}); cf.\
\cite[p.~5]{cohn-66}. Each pair
  \begin{equation}\label{eq:E-Produkt}
     (a,b):=(1,0)\cdot E(t_{n})\cdot E(t_{n-1})\cdots E(t_1)
  \end{equation}
with $t_1, t_2,\ldots,t_{n}\in R$ and $n\geq 0$ is admissible. A point of
$\PP(R)$ is in the connected component of $R(1,0)$ if, and only if, it has a
representative $(a,b)$ of this form; see \cite[Theorem~3.2]{blu+h-00c}.

Suppose now that $(a,b)\in R^2$ is given according to
(\ref{eq:E-Produkt}). From (\ref{eq:U-perp}), (\ref{eq:E-inv}), and
$E(0)^2=-I$, where $I$ denotes the identity in $\GL_2(R)$, the point
$R(a,b)^\iota$ is represented by
\begin{equation}\label{eq:E-Produkt-iota}
   \SMat2{v\\w}:=(-I)^{n-1} E(0)\cdot E(-t_{1})\cdot E(-t_{2})\cdots
E(-t_{n})\cdot
    E(0)\cdot\SMat2{0\\1}.
\end{equation}
Clearly, the irrelevant factor $(-I)^{n-1}$ may be omitted. In
particular, this includes formulae (\ref{eq:iota-0}) and
(\ref{eq:iota-1}) by letting $n=0$ and $n=1$, respectively. On the
other hand, for $n=2,3$ we get from (\ref{eq:E-Produkt}) and
(\ref{eq:E-Produkt-iota})
\begin{eqnarray}
   (R(t_2t_1-1,t_2))^\iota &=&
   (-t_2,t_1t_2-1)^\T R,\label{eq:iota-2}\\
   (R(t_3 t_2 t_1-t_3-t_1,t_3 t_2 - 1))^\iota &=&
   (-t_2 t_3 + 1,t_1 t_2 t_3 - t_1 - t_3 )^\T R\label{eq:iota-3}
\end{eqnarray}
for all $t_1,t_2,t_3\in R$.

If the connected component of $R(1,0)$ has finite diameter $m$ then
each of its points has a representative of the form
(\ref{eq:E-Produkt}) with $0\leq n\leq m$. See
\cite[formula~(10)]{{blu+h-00c}}. Also, from $E(0)^2=-I$ and
$(1,0)\cdot E(t)=(1,0)\cdot E(1)\cdot E(t+1)$ for all $t\in R$, it is
enough to consider products where $n=\max\{2,m\}$.

The explicit formula (\ref{eq:iota-2}) describes the $\iota$-images
of \emph{all} points of $\PP(R)$ if $R$ is a ring of \emph{stable
rank} $2$, since here $\PP(R)$ is connected and $m\leq 2$. See
\cite[Proposition 1.4.2]{herz-95} and \cite[Example~5.2
(b)]{blu+h-00c}. We add in passing that the stable rank of $R$ equals
the stable rank of its opposite ring $R^\circ$ \cite[2.2]{veld-85}.
See Example \ref{exa:endo} below for an application of formula
(\ref{eq:iota-3}).
\end{rem}

\section{Compatibility and Dual Compatibility}

We consider a chain geometry $\Sigma(K,R)$. For a fixed point
$p\in\PP(R)$ the set $\PP(R)_p$ consists of all points distant from
$p$, and $\fC(K,R)_p$ consists of all sets $\cD\setminus \{p\}$,
where $\cD$ is a chain through $p$. An element of $\fC(K,R)_p$ will
be called a \emph{block}. Altogether the \emph{residue of
$\Sigma(K,R)$ at $p$} is the incidence structure
\begin{equation}\label{eq:residuum}
  \Sigma(K,R)_p=(\PP(R)_p,\fC(K,R)_p).
\end{equation}
Cf.\ \cite[Section~4]{blu+h-99a}.

Let $\infty:=R(1,0)$. The chains $\cD_1,\cD_2$ through $\infty$ are
called \emph{compatible at} $\infty$ if they belong to the same orbit
under the action of the group
\begin{equation}\label{eq:delta}
   \Delta:=\{\SMat2{a&0\\c&1}\mid a\in R^*, c\in R\}\subset\GL_2(R)
\end{equation}
on $\fC(K,R)$. Then also the blocks $\cD_1\setminus\{\infty\}$ and
$\cD_2\setminus\{\infty\}$ of $\Sigma(K,R)_\infty$ will be called
compatible. By definition, the compatibility of chains (at a common
point) is a $\GL_2(R)$-invariant notion (see \cite[Section
3]{blu+h-99a}), whence one has a compatibility relation on the set of
blocks of each residue $\Sigma(K,R)_p$.

It suffices to consider the case where $p=\infty$. A point $R(a,b)$ is distant
from $\infty$ exactly if $b$ is a unit in $R$. The bijection
\begin{equation}\label{eq:R}
    \PP(R)_\infty\to R : R(x,1)\mapsto x
\end{equation}
will be used to identify $\PP(R)_\infty$ with $R$. By
\cite[Theorem~4.1]{blu+h-99a}, a subset of $\fC(K,R)_\infty$ is a
\emph{compatibility class} exactly if it has the form
\begin{equation}\label{eq:comp-class}
  \{(u^{-1}K u)a+c\mid a\in R^*,\;c\in R\} \mbox { with } u\in R^*.
\end{equation}
Recall that a \emph{partial affine space} is an incidence structure
resulting from an affine space by removing certain parallel classes
of lines (but no points). If $K^*$ is not normal in $R^*$ then the
residue $\Sigma(K,R)_\infty$ cannot be embedded in any affine space,
since the points $0,1\in R$ are joined by more than one block, namely
by all subfields $u^{-1}K u$, where $u\in R^*$. However, the point
set $\PP(R)_\infty$ together with one compatibility class
(\ref{eq:comp-class}) forms a partial affine space which extends to
the affine space $\Aa(u^{-1}K u,R)$ on the \emph{left vector space}
$R$ over $u^{-1}K u$; see \cite[Theorem~4.2]{blu+h-99a}.

The construction described above can be carried over to the dual
chain geometry $\widehat\Sigma(K,R)$. We restrict ourselves to the
residue of $\widehat\Sigma(K,R)$ at $\infty^\iota=(0,1)^\T R$, where
$\iota$ is the canonical isomorphism. The counterpart of (\ref{eq:R})
is the bijection
\begin{equation}\label{eq:Rdach}
    \Pd(R)_{\infty^\iota}\to R : \SMat2{-1\\x} R \mapsto x.
\end{equation}

Two chains of $\widehat\Sigma(K,R)$ through $(1,0)^\T R$ are
compatible at $(1,0)^\T R$, if they belong to the same orbit with
respect to the group $\Delta^\T:=\{D^\T\mid D\in\Delta\}$ acting on
$\Pd(R)$ from the left; cf.\ (\ref{eq:delta}).  So the compatibility
in $\widehat\Sigma(K,R)_{\infty^\iota}$ is governed by the group
\begin{equation}\label{eq:deltadach}
   \widehat\Delta:=M\cdot\Delta^\T\cdot M^{-1} =
   \{\SMat2{1&0\\c&d}\mid c\in R, d\in R^*\}
   \subset\GL_2(R),
\end{equation}
where $M\in\GL_2(R)$ is any matrix taking $(1,0)^\T R$ to
$\infty^\iota=(0,1)^\T R$. The partial affine spaces defined by the
compatibility classes  in $\widehat\Sigma(K,R)_{\infty^\iota}$ are embedded in
affine spaces $\widehat\Aa(u K u^{-1},R)$, where $R$ is considered as
\emph{right vector space} over $u K u^{-1}$.

Let $\cD_1,\cD_2$ be chains of $\Sigma(K,R)$ with common point $p$.
We say that $\cD_1,\cD_2$ are \emph{dually compatible at} $p$ if, and
only if, $\cD_1^\iota,\cD_2^\iota\in\widehat\fC(K,R)$ are compatible
at $p^\iota$. Analogously, we define dual compatibility of blocks of
a residue.

\begin{thm}\label{th:vergleich}
Suppose that\/ $\PP(R)_\infty$ and\/ $\Pd(R)_{\infty^\iota}$ are
identified with the ring $R$ according to (\ref{eq:R}) and
(\ref{eq:Rdach}), respectively. Then the following holds:
\begin{enumerate}\itemsep=0cm
  \item Each point of\/ $\PP(R)_\infty$ and its $\iota$-image are the same.
  \item The residue of\/ $\Sigma(K,R)$ at $\infty$ coincides with
        the residue of\/ $\widehat\Sigma(K,R)$ at\/ $\infty^\iota$.
  \item The equivalence relations of ``compatibility'' and ``dual
        compatibility'' on the set of blocks are the same exactly if
        the multiplicative group $K^*$ is normal in the multiplicative group $R^*$.
\end{enumerate}
\end{thm}
\begin{proof}
(a) This is obviously true.

(b) In both residues the blocks are exactly the sets $dKa +c$ with $a,d\in R^*$
and $c\in R$.

(c) Suppose that $K^*$ is not normal in $R^*$. Then there is a $u\in R^*$ with
$u K\neq K u$. The compatibility class of the block $K$ contains exactly the
blocks $K a+c$ with $a\in R^*$ and $c\in R$. The only block of this class
running through $0$ and $u$ is $K u$. We read off from $0,u\in u K\neq K u$
that the block $u K$ is not compatible to $K$. However, the dual compatibility
class of $K$ contains $u K$. Therefore the relations are different. On the
other hand, if $K^*$ is normal in $R^*$ then all blocks are compatible and
dually compatible; see \cite[Theorem~4.2]{blu+h-99a}. This completes the proof.
\end{proof}

\begin{rem}\label{rem-comp}
From Theorem \ref{th:vergleich}(c) and formula (\ref{eq:U-perp}) we
obtain the following: The canonical isomorphism
$\iota:\Sigma(K,R)\to\widehat\Sigma(K,R)$ preserves compatibility (at
all points) if, and only if, $K^*$ is normal in $R^*$. In particular,
this shows that the notion of compatibility needs not be invariant
under isomorphisms of chain geometries.
\end{rem}

Let $S_\infty$ be the set of all $\fL\subset\fC(K,R)_\infty$ such
that $(\PP(R)_\infty,\fL)$ is a partial affine space. We have seen
before that each (dual) compatibility class of blocks belongs to
$S_\infty$. From \cite[Lemma 2.1]{blu+h-99a} and \cite[Proposition
2.2]{blu+h-99a}, two distinct points $R(x,1)$, $R(y,1)$ of the
residue are joined by at least one block exactly if they are distant.
This is equivalent to $y-x\in R^*$. From (\ref{eq:comp-class}), the
set of blocks through two distant points of $\PP(R)_\infty$ has
exactly one element in common with each (dual) compatibility class.
This means that each (dual) compatibility class is a maximal element
of $S_\infty$ with respect to inclusion. One could conjecture that
the maximal elements of $S_\infty$ were exactly the (dual)
compatibility classes. However, there may also be other maximal
elements of $S_\infty$:

\begin{exa}\label{exa:mix}
Let $R=\HH$ be the field of real quaternions with the usual $\RR$-basis
$\{1,i,j,k\}$. Further, let $K=\CC=\RR+\RR i$ be a subfield of complex numbers.
The blocks of $\fC(\CC,\HH)_\infty$ compatible to $\CC$ are exactly the lines
of the complex affine plane $\Aa(\CC,\HH)$. Put
\begin{equation}\label{eq:baer}
 \cB:=\{ a (\RR+\RR j) + c
 \mid
 a\in\CC^*,\;c\in\HH\}.
\end{equation}
Each element of $\cB$ is a block, since $\RR+\RR
j=(1+k)^{-1}\CC(1+k)$, but not a line of $\Aa(\CC,\HH)$. Obviously,
the elements of $\cB$ are Baer subplanes of $\Aa(\CC,\HH)$.  We apply
the well known procedure of \emph{derivation}: All lines of
$\Aa(\CC,\HH)$ that are parallel to a line of the Baer subplane
$\RR+\RR j$ are removed and instead the Baer subplanes belonging to
$\cB$ are introduced as ``new lines''. This gives a
(non-desarguesian) affine plane with point set $\HH$. Cf.\
\cite[Theorem~3.14]{john-2000}. By construction, the set of lines of
the derived plane is a maximal element of $S_\infty$, but it is
neither a compatibility class nor a dual compatibility class.

A reader who is familiar with \emph{elliptic geometry} will easily
verify the following: The lines of the projective $3$-space
$\PP(\RR,\HH$) (which carries the structure of an \emph{elliptic
space} coming from the Euclidean norm on quaternions) are exactly the
blocks of $\fC(\CC,\HH)_\infty$ through $0\in\HH$. Two blocks $B_1$,
$B_2$ through $0$ are compatible exactly if there is a \emph{right
Clifford translation} ($\RR x\mapsto \RR x a$, $a\in\HH^*$) of
$\PP(\RR,\HH)$ taking $B_1$ to $B_2$. This characterizes the lines
$B_1$ and $B_2$ as \emph{left Clifford parallel}. Similarly, dual
compatibility corresponds to \emph{right} Clifford parallelism. The
set of blocks through $0$ that are compatible to $\CC$ appears as a
regular spread (elliptic linear congruence of lines) of the elliptic
space (fig.~\ref{abb1}). All lines of this spread are left parallel.
See, among others, \cite[Chapter VII]{cox-78} and
\cite[p.~76]{kar+k-88}. It is well known that here the process of
derivation means that one regulus $\cR$ of this spread is replaced
with its opposite regulus $\cR^\circ$ (fig.~\ref{abb2}); see, e.g.,
\cite[p.~101--102]{bruck+b-64}. The lines of $\cR^\circ$ are mutually
right Clifford parallel, since $\cB$ is a subset of a dual
compatibility class.
{\unitlength1cm
      \begin{center}
      \begin{minipage}[t]{4cm}
         \begin{picture}(4.0,5.11)
         \put(0,0)
         {\includegraphics[width=4.0cm]{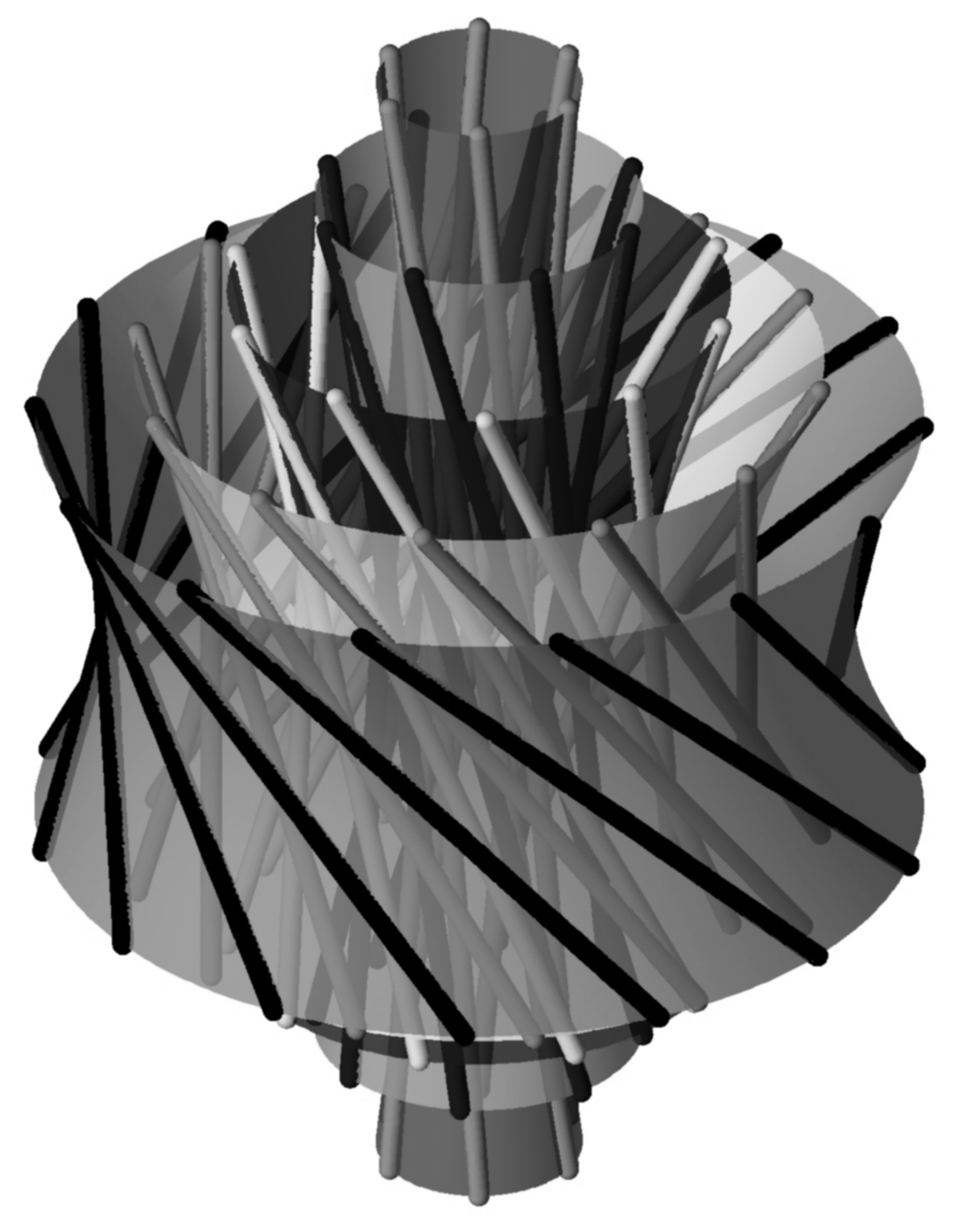}}
         \end{picture}
         {\refstepcounter{abbildung}\label{abb1}
          \centerline{Fig.~\ref{abb1}.}}
         \end{minipage}
         \hspace{1.5cm}
      \begin{minipage}[t]{4cm}
         \begin{picture}(4.0,5.11)
         \put(0,0)
         {\includegraphics[width=4.0cm]{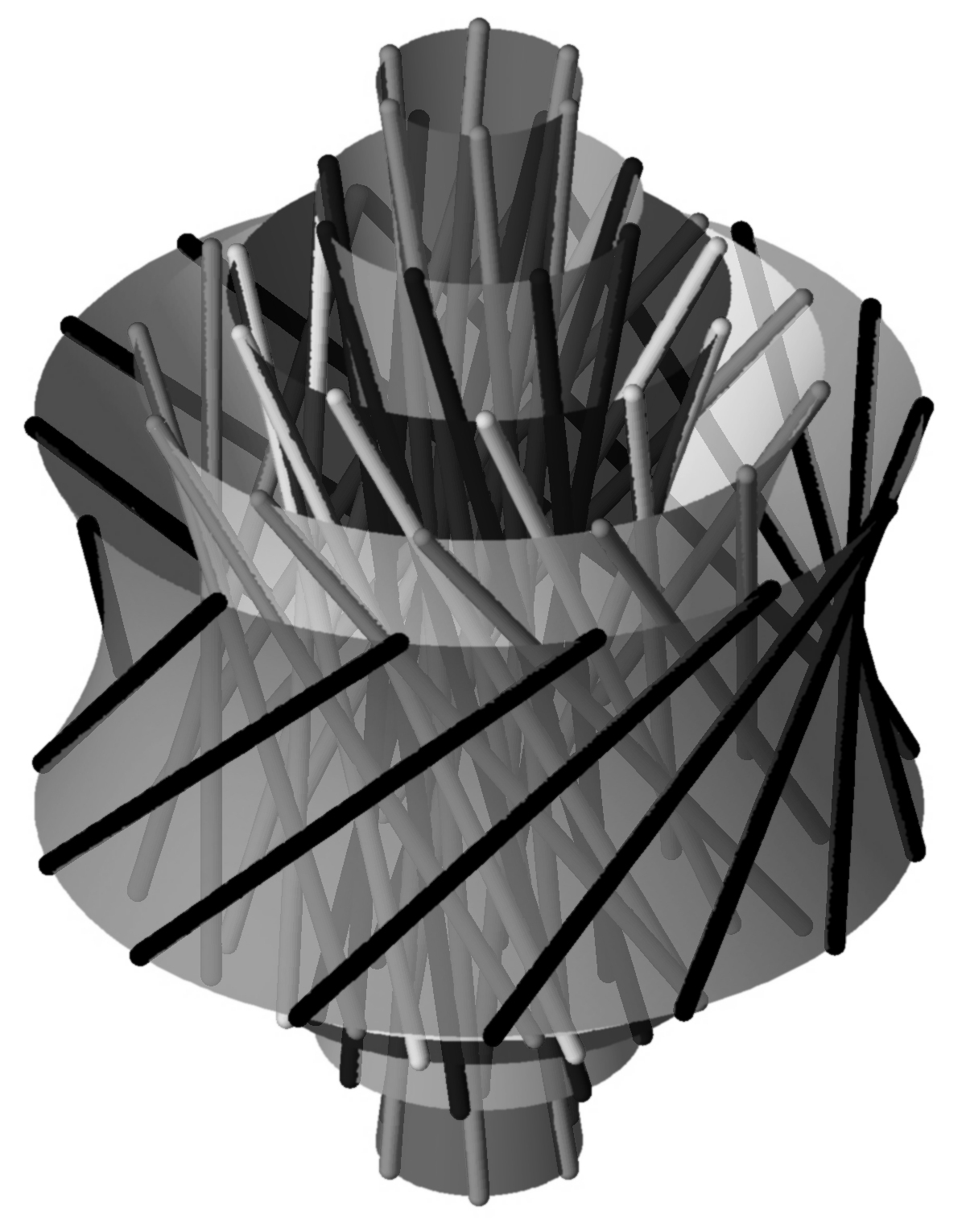}}
        \end{picture}
        {\refstepcounter{abbildung}\label{abb2}
          \centerline{Fig.~\ref{abb2}.}}
      \end{minipage}
      \end{center}
}%
\end{exa}

\section{Isomorphisms}

In this section we consider two chain geometries $\Sigma(K,R)$ and
$\Sigma(K',R')$.

\begin{rem}\label{rem-ringiso}
For each $u'\in {R'}^*$ we have $\Sigma(K',R') =
\Sigma({u'}^{-1}K'u',R')$ by virtue of the transformation of
$\PP(R')$ given by the matrix $\diag({u'},{u'})$. So, if $\phi:R\to
R'$ is an isomorphism of rings such that $K^\phi = {u'}^{-1}K'u'$
holds for a suitable $u'\in {R'}^*$, then
\begin{equation}\label{eq:phiquer}
\overline\phi:\PP(R)\to\PP(R'):R(a,b)\mapsto R'(a^\phi,b^\phi)
\end{equation}
is an isomorphism $\Sigma(K,R)\to\Sigma(K',R')$ mapping
$\infty=R(1,0)$ to $\infty':=R'(1',0')$.

The restriction  of  $\overline\phi$  to ${\PP(R)_\infty}$ is an
isomorphism from the residue $\Sigma(K,R)_\infty$ onto
$\Sigma(K',R')_{\infty'}$. According to the identification
(\ref{eq:R}), applied to ${\PP(R)_\infty}$ and $\PP(R')_{\infty'}$,
this restriction coincides with $\phi$. Using (\ref{eq:comp-class}),
one sees that $\phi$ preserves compatibility of blocks. The same
holds for the restriction of $\overline\phi$ to any other residue,
because the actions of $\GL_2(R)$ and $\GL_2(R')$ on $\PP(R)$ and
$\PP(R')$, respectively, are isomorphic via $\overline\phi$.
Altogether, $\overline\phi$ preserves compatibility of chains.
\end{rem}
 We now study the case of antiisomorphisms.

\begin{thm}\label{thm:iso}
Let $\phi:R\to R'$ be an  antiisomorphism of rings such that $K^\phi
= {u'}^{-1}K'u'$ for some $u'\in {R'}^*$. Then the product of the
canonical isomorphism $\iota:\PP(R)\to\Pd(R)$ and the mapping
\begin{equation}\label{eq:phidach}
\widehat\phi:\Pd(R)\to \PP(R'): \SMat2{v\\w}R
  \mapsto R'(v^\phi,w^\phi).
\end{equation}
is an isomorphism of\/ $\Sigma(K,R)$ onto\/ $\Sigma(K',R')$.
\end{thm}
\begin{proof}
The antiisomorphism  $\phi:R\to R'$ is an isomorphism $R^\circ\to
R'$. So, from Remark \ref{rem-ringiso}, the mapping $\widehat\phi$ is
an isomorphism of $\widehat\Sigma(K,R)$ onto $\Sigma(K',R')$, whence
$\iota\widehat\phi$ has the required properties.
\end{proof}

\begin{rem}
We conclude from Remark \ref{rem-comp} and Remark \ref{rem-ringiso}
that the isomorphism $\iota\widehat\phi$ preserves compatibility if,
and only if, $K^*$ is normal in $R^*$.
\end{rem}

Theorem \ref{thm:iso} does not give an explicit description of the
isomorphism $\iota\widehat\phi$ from $\Sigma(K,R)$ onto
$\Sigma(K',R')$, since we did not describe the canonical isomorphism
$\iota$ explicitly either. As in Remark \ref{rem-iota}, we know more
for certain points:

\begin{rem}\label{rem-sigma}
Let $\phi:R\to R'$ be given as in Theorem \ref{thm:iso}. For
$M\in\GL_2(R)$ let $M^\phi$ be the matrix in $\GL_2(R')$ obtained by
applying $\phi$ to the entries of $M$. We observe that $M\mapsto
(M^\T)^\phi$ is an antiisomorphism of groups. Also, for each point
$q\in\Pd(R)$ and each matrix $M\in\GL_2(R)$ we have
\begin{equation}\label{eq:M-phi}
  (M\cdot q)^{\widehat\phi}=q^{\widehat\phi}\cdot (M^\T)^\phi.
\end{equation}
The product $\iota\widehat\phi$ is an isomorphism of $\Sigma(K,R)$
onto $\Sigma(K',R')$. However, by (\ref{eq:iota-0}) and
(\ref{eq:phidach}), it takes $R(1,0)$ to $R'(0',1')$ rather than to
$R'(1',0')$. So let $\eta: R'(a',b')\mapsto R'(b',-a')$ be the
transformation of $\PP(R')$ induced by $E(0')^{-1}$. We focus our
attention to the isomorphism
\begin{equation}\label{eq:sigma}
  \sigma:=\iota\widehat\phi\eta
\end{equation}
of $\Sigma(K,R)$ onto $\Sigma(K',R')$. By construction,
\begin{equation}\label{eq:sigma-0}
    (R(1,0))^\sigma = R'(1',0').
\end{equation}
We aim at an explicit computation of the $\sigma$-images of all points in the
connected component of $R(1,0)$: Let $p=R(a,b)$ with $(a,b)$ as in
(\ref{eq:E-Produkt}), whence $p^\iota=(v,w)^\T R$ with $(v,w)^\T$ as in
(\ref{eq:E-Produkt-iota}). Using (\ref{eq:M-phi}) and $E(-t)^\T=(-I)\cdot E(t)$
we obtain from (\ref{eq:E-Produkt-iota}) and (\ref{eq:sigma}) that
\begin{equation}\label{eq:Produkt-sigma}
   p^\sigma=R'(1',0')\cdot
  E(t_n^\phi)\cdot E(t_{n-1}^\phi)\cdots E(t_1^\phi).
\end{equation}

In particular, we have
\begin{eqnarray}
    (R(t_1,1))^\sigma &=& R'(t_1^\phi,1'),\label{eq:sigma-1}\\
   (R(t_2t_1-1,t_2))^\sigma &=&
   R'(t_2^\phi t_1^\phi-1',t_2^\phi),\label{eq:sigma-2}\\
   R((t_3 t_2 t_1-t_3-t_1,t_3 t_2 - 1))^\sigma &=&
   R'(t_3^\phi t_2^\phi t_1^\phi-t_3^\phi-t_1^\phi,
      t_3^\phi t_2^\phi-1')\label{eq:sigma-3}
\end{eqnarray}
for all $t_1,t_2,t_3\in R$, as counterparts of formulae (\ref{eq:iota-1}),
(\ref{eq:iota-2}), and (\ref{eq:iota-3}).

From \cite[Theorem 2.4]{bart-89}, for rings of stable rank $2$ an
isomorphism of $\Sigma(K,R)$ onto $\Sigma(K',R')$ can be
\emph{defined} according to (\ref{eq:sigma-2}), even when $\phi:R\to
R'$ is a \emph{Jordan isomorphism} satisfying certain conditions on
the image of $K$. See also \cite{blunck-94}, \cite[9.1]{herz-95}, and
\cite{maeu+m+n-80} for related results.
\end{rem}

\begin{exa}\label{exa:endo}
Let $R=\End_K(V)$ be the endomorphism ring of an infinite dimensional
vector space $V$ over a commutative field $K$. For each $a\in R$ the
transpose mapping $a^\T$ is an endomorphism of the dual vector space.
We put $R':=\{a^\T\mid a\in R\}$ so that $\phi:R\to R':a\mapsto a^\T$
is an antiisomorphism of rings. The field $K=:K'$ can be embedded in
$R$ via $k\mapsto k\cdot\id_V$ and in $R'$ via $k\mapsto
k\cdot(\id_V)^\T$ ($k\in K$), whence $K^\phi=K'$. Then an isomorphism
$\sigma$ of the corresponding chain geometries is given by
(\ref{eq:phidach}) and (\ref{eq:sigma}). From
\cite[Theorem~5.3]{blu+h-00c}, the projective line $\PP(R)$ is
connected and its diameter equals $3$. So formula (\ref{eq:sigma-3})
can be used to calculate the $\sigma$-images of \emph{all} points.
Further, the canonical isomorphism
$\Sigma(K,R)\to\widehat\Sigma(K,R)$ is given by (\ref{eq:iota-3}).
\end{exa}


\begin{thebibliography}{10}

\bibitem{bart-89}
C.G.~Bartolone.
\newblock Jordan homomorphisms, chain geometries and the fundamental theorem.
\newblock {\em Abh.\ Math.\ Sem.\ Univ.\ Hamburg}, 59:93--99, 1989.

\bibitem{blunck-94}
A.~Blunck.
\newblock Chain spaces over {J}ordan systems.
\newblock {\em Abh.\ Math.\ Sem.\ Univ.\ Hamburg}, 64:33--49, 1994.



\bibitem{blu+h-99a}
A.~Blunck and H.~Havlicek.
\newblock Extending the concept of chain geometry.
\newblock {\em Geom.\ Dedicata}, 83:119--130, 2000.

\bibitem{blu+h-00a}
A.~Blunck and H.~Havlicek.
\newblock Projective representations {I}. {P}rojective lines over rings.
\newblock {\em Abh.\ Math.\ Sem.\ Univ.\ Hamburg}, 70:287--299, 2000.

\bibitem{blu+h-00c}
A.~Blunck and H.~Havlicek.
\newblock The connected components of the projective line over a ring.
\newblock {\em Adv.\ in Geometry} (to appear).

\bibitem{bour-89}
N.~Bourbaki.
\newblock {\em Elements of Mathematics, Algebra I}.
\newblock Springer, Berlin Heidelberg New York, 1989.

\bibitem{bruck+b-64}
R.H.~Bruck and R.C.~Bose.
\newblock The construction of translation planes from projective spaces.
\newblock {\em J.\ Algebra}, 1:85--102, 1964.

\bibitem{cohn-66}
P.M.~Cohn.
\newblock On the structure of the {${\rm GL}_2$} of a ring.
\newblock {\em Inst.\ Hautes Etudes Sci.\ Publ.\ Math.}, 30:5--53, 1966.

\bibitem{cox-78}
H.S.M.~Coxeter.
\newblock {\em Non-{E}uclidean Geometry}.
\newblock University of Toronto Press, Toronto, 5th edition, 1978.

\bibitem{herz-95}
A.~Herzer.
\newblock Chain geometries.
\newblock In F.~Buekenhout, editor, {\em Handbook of Incidence Geometry}.
  Elsevier, Amsterdam, 1995.

\bibitem{john-2000}
N.L.~Johnson.
\newblock {\em Subplane Covered Nets}.
\newblock Dekker, New York, 2000.

\bibitem{kar+k-88}
H.~Karzel and H.-J.~Kroll.
\newblock {\em Geschichte der Geometrie seit Hilbert}.
\newblock Wiss.\ Buchges., Darmstadt, 1988.

\bibitem{maeu+m+n-80}
H.~M\"aurer, R.~Metz, and W.~Nolte.
\newblock Die {A}utomorphismengruppe der {M}\"obiusgeometrie einer
  {K}\"orpererweiterung.
\newblock {\em Aequationes Math.}, 21:110--112, 1980.

\bibitem{veld-85}
F.D.~Veldkamp.
\newblock Projective ring planes and their homomorphisms.
\newblock In R.~Kaya, P.~Plaumann, and K.~Strambach, editors, {\em Rings and
  Geometry}. D.\ Reidel, Dordrecht, 1985.

\end{thebibliography}

\noindent Institut f\"ur Geometrie\\ Technische Universit\"at\\ Wiedner
Hauptstra{\ss}e 8--10\\ A-1040 Wien\\Austria

\end{document}